\documentclass[conference]{IEEEtran}
\IEEEoverridecommandlockouts
\usepackage{cite}
\usepackage{amsmath,amssymb,amsfonts,amsthm,mathtools}
\usepackage{algorithmic}
\usepackage{graphicx}
\usepackage{textcomp}
\usepackage{xcolor}
\usepackage{tikz}
\def\BibTeX{{\rm B\kern-.05em{\sc i\kern-.025em b}\kern-.08em
    T\kern-.1667em\lower.7ex\hbox{E}\kern-.125emX}}

\newcommand{\drm}{\mathrm{d}}
\newcommand{\euler}{\mathrm{e}}

\newcommand{\id}{\mathrm{Id}}
\newcommand{\ii}{\mathrm{i}}
\newcommand{\ran}{\mathrm{ran}}
\newcommand{\supp}{\mathrm{supp}}

\newcommand{\obs}{\mathrm{obs}}

\newcommand{\NN}{\mathbb{N}}
\newcommand{\RR}{\mathbb{R}}
\newcommand{\CC}{\mathbb{C}}

\newcommand{\cS}{\mathcal{S}}
\newcommand{\cL}{\mathcal{L}}
\newcommand{\cN}{\mathcal{N}}

\newcommand{\1}{\mathbf{1}}

\DeclareMathOperator{\re}{\mathrm{Re}}

\newtheorem{theorem}{Theorem}[section]
\newtheorem{proposition}[theorem]{Proposition}
\newtheorem{lemma}[theorem]{Lemma}
\newtheorem{corollary}[theorem]{Corollary}
\theoremstyle{definition}

\newtheorem{definition}[theorem]{Definition}
\theoremstyle{remark}

\mathtoolsset{showonlyrefs}

\begin{document}

\title{Stabilizability of neural fields from thick subsets \\
\thanks{CB was funded by the Deutsche Forschungsgemeinschaft (DFG, German Research Foundation) – Project number 283135041 awarded to MR. \\ \copyright 2026 IEEE. Personal use of this material is permitted. Permission from IEEE must be obtained for all other uses, in any current or future media, including reprinting/republishing this material for advertising or promotional purposes, creating new collective works, for resale or redistribution to servers or lists, or reuse of any copyrighted component of this work in other works.}
}

\author{\IEEEauthorblockN{Clemens Bombach}
\IEEEauthorblockA{\textit{Predictive Analytics} \\
\textit{TU Chemnitz}\\
Chemnitz, Germany \\
clemens.bombach@hsw.tu-chemnitz.de}
\and
\IEEEauthorblockN{Marco Ragni}
\IEEEauthorblockA{\textit{Predictive Analytics} \\
\textit{TU Chemnitz}\\
Chemnitz, Germany \\
marco.ragni@hsw.tu-chemnitz.de}
}

\maketitle

\begin{abstract}
An important problem in neuro-engineering is the stabilization of neural fields. In applications, it is often assumed that the actuator placement can be chosen arbitrarily. In this work, we investigate the stabilizability of controlled Amari-type neural fields where the control input is prescribed to act only on a fixed subset of the neural field. We show that the linearized neural field is open-loop stabilizable under suitable assumptions on the interaction strength and a mild relative density assumption on the control set. Our geometric assumption requires that the volume of each cube intersected with the control set must be bounded below. As a consequence, we derive closed-loop stabilizability of the neural fields, under sensor/actuator placement constraints. Numerical simulations are used to illustrate the results and obtain empirical estimates on the control cost.
\end{abstract}


\section{Introduction}
The Amari-type neural field equation, introduced in \cite{amari1977dynamics}, is a semilinear integro-differential equation that describes the time evolution of activity of a population of neurons. Denote by $d \in \NN$ the spatial dimension\footnote{In applications, we have that $d \in \{1,2,3\}$.} of the neural field and by $p \in (1,\infty)$ an integrability index.\footnote{We refer to Section \ref{sec:prelim} for an overview on mathematical notation.}  In this work, we consider a \emph{controlled} Amari-type field equation given by
\begin{align}\label{eq:amari}
	(\partial_t  + \alpha)a(t;x) - \mu \int\limits_{\RR^d} k(x-y)f(a(t;y))\,\drm y &= \1_E u(t;x) \notag \\ a(0;x) &= a_0(x)
\end{align}
where $a \in L^\infty((0,\infty);L^p(\RR^d))$ denotes the difference in neural activity compared to the base activity $0$, the function $a_0 \in L^p(\RR^d)$ is an initial activity and $u \in L^\infty((0,\infty);L^p(\RR^d))$ an input acting as a control function that may only affect $a$ via the control set $E$. The transfer function $f \in C^1(\RR)$ and the kernel $k \in \cS(\RR^d)$ describe how the neural response spreads after activation. Typical choices for $f$ and $k$ include a sigmoid shape $\tanh$ function and a Mexican hat wavelet, respectively. The parameters $\alpha,\mu > 0$ measure the strength of the activity decay and the neural interaction. As a semilinear equation, it remains mathematically tractable while being biologically grounded. For an overview on the mathematical and physiological theory of neural fields we refer to \cite{cook2022neural, coombes2014neural}.

The study of control and stability properties of neural field models such as \eqref{eq:amari} is motivated by applications in neuroscience and neuro-engineering, such as the prediction of visual illusions \cite{tamekue2024mathematical, bolelli2025neural} and the therapy of epilepsy \cite{taylor2014computational} or Parkinson's disease \cite{chaillet2017robust}. The exact controllability of \eqref{eq:amari} has been proven in the works \cite{tamekue2024mathematical, tamekue2025control} whereas the stabilizability of a version of \eqref{eq:amari} with delay and multiple populations of neurons has been proven in \cite{brivadis2024adaptive}.

The above works assume that $E = \RR^d$, i.e., that we may choose the support of the control function $u$ without restriction. In applications, it may happen that there are uncontrollable or unobservable regions of the underlying domain on which the control cannot act directly. Practical reasons for such a limitation may include physical obstructions regarding the placement of sensors or actuators. This restricts the way in which the control function $u$ may act. For studies on these issues for structural brain networks, including a discussion of the effect of the control set geometry, we refer to \cite{gu2015controllability, karrer2020practical}. We note that structural brain networks, unlike Amari-type neural fields, represent linear, space- and time-discrete models of human brain activity. Therefore, it is natural to consider limitations on the control set geometry in the continuum case.

In this work, we therefore investigate the stabilization properties of \eqref{eq:amari} under the weaker assumption that the control set $E$ is a thick subset of $\RR^d$. Intuitively, this means that the set $E$ may not contain any gaps of large $d$-dimensional volume -- see Definition \ref{def:thick} for a precise definition. Typical examples of thick sets include equidistributed sets of balls or discs. An example is pictured in Figure \ref{fig:thick}.  We will show that the system \eqref{eq:amari} is closed-loop stabilizable in $L^2$ by a suitable feedback law provided the coupling constant $\mu$ is not too large and $f^\prime (\theta)- f^\prime(0)$ is small for all $\theta \in \RR$.

\begin{figure}\label{fig:thick} \centering
	\begin{tikzpicture}
		\draw[thin,gray] (0,0) grid (4,4);
		\filldraw[color=violet](0.5,0.5) circle (0.25);
		\filldraw[color=violet](1.4,0.3) circle (0.25);
		\filldraw[color=violet](2.4,0.4) circle (0.25);
		\filldraw[color=violet](3.4,0.6) circle (0.25);
		\filldraw[color=violet](0.4,1.6) circle (0.25);
		\filldraw[color=violet](1.3,1.7) circle (0.25);
		\filldraw[color=violet](2.5,1.5) circle (0.25);
		\filldraw[color=violet](3.3,1.7) circle (0.25);
		\filldraw[color=violet](0.3,2.6) circle (0.25);
		\filldraw[color=violet](1.3,2.65) circle (0.25);
		\filldraw[color=violet](2.7,2.4) circle (0.25);
		\filldraw[color=violet](3.4,2.6) circle (0.25);
		\filldraw[color=violet](0.4,3.3) circle (0.25);
		\filldraw[color=violet](1.5,3.4) circle (0.25);
		\filldraw[color=violet](2.7,3.3) circle (0.25);
		\filldraw[color=violet](3.5,3.5) circle (0.25);
		\draw[black, very thick] (0.4,0.4) rectangle (2.4,2.4);
	\end{tikzpicture}
	\caption{An example of a thick set, drawn in purple. No matter where the black rectangle is moved, it captures at least one disc and therefore a positive amount of volume.}
\end{figure}
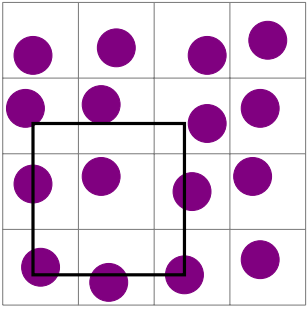

Our analysis of the nonlinear system \eqref{eq:amari} rests on a careful study of the linearized system, which we write as the abstract Cauchy problem
\begin{equation}\label{eq:amari_lin}
	\partial_t a - Ba =  \1_E u, \quad a(0) = a_0\,,
\end{equation}
where $B = -\alpha \id + \mu K$ with $\id$ denoting the identity operator and $K:L^p(\RR^d) \to L^p(\RR^d)$ given as the convolution operator $Ka = k * a$. For small $u$, the linearized system behaves similarly to the nonlinear one, see, e.g. \cite{bolelli2025neural}.


The unique mild solution $a$ of \eqref{eq:amari_lin} is given by the Duhamel formula
\begin{equation}\label{eq:amari_lin_mild}
	a(t) = S_t a_0 + \int\limits_0^t S_{t-s}\1_E u(s)\,\drm s, \quad (t \in (0,\infty))
\end{equation}
where $S_t = \euler^{tB}$ is the strongly continuous semigroup generated by the bounded linear operator $B$. This semigroup encodes the dynamics of the uncontrolled, linearized system. We will demonstrate that \eqref{eq:amari_lin} is open-loop stabilizable, assuming that $E$ is thick and $\mu$ is not too large. In the case where $p = 2$, this allows one to deduce the existence of a feedback law $F$ that makes the system \eqref{eq:amari_lin} closed-loop stabilizable. We will then employ this feedback law to stabilize the nonlinear equation.

The main technical difficulty arises in the study of the linearized system. This can be explained as follows: Our strategy for proving the open-loop stabilizability of \eqref{eq:amari_lin} rests on the well-known duality between controllability or stabilizability on the one side and observability on the other side. More precisely, we verify an abstract criterion for open-loop stabilizability proven in \cite{egidi2024sufficient}. This criterion follows from the fact that the open-loop stabilizability of \eqref{eq:amari_lin} is equivalent to a weak observability property of the adjoint system
\begin{equation*}
	\partial_t \phi - B^* \phi = 0, \quad \psi = \1_E \phi, \quad \phi(0) = \phi_0\,.
\end{equation*}
Roughly speaking, the result of \cite{egidi2024sufficient} states that \eqref{eq:amari_lin} is open-loop stabilizable if a certain abstract uncertainty principle and dissipation estimate hold. This criterion can be seen as a generalization of the Lebeau-Robbiano strategy for proving null-controllability of the heat equation, see \cite{lebeau1995controle} and, for instance, the related works \cite{miller2010direct, gallaun2020sufficient, bombach2023observability}.

In our concrete situation, the relevant uncertainty principle is the Logvinenko-Sereda theorem, which states that if a function $u$ is band-limited, i.e. if the Fourier transform $\hat u$ is concentrated on a cube, then $u$ cannot be concentrated on the complement of a thick set in a quantitative sense.

The corresponding dissipation estimate requires that high frequencies in the Fourier decomposition of $S_tu$ must decay sufficiently quickly as $t \to \infty$. If the semigroup $S_t$ were smoothing -- as is the case, for example, when considering the heat semigroup -- then this would be a straightforward application of the argument given in \cite{bombach2023observability}, see also \cite{egidi2024sufficient} for an application to the semigroup generated by the fractional Laplacian. The main complication in the proof of stabilizability for the Amari-type field equation comes from the fact that the semigroup $S_t$ is not smoothing. Thus, a careful analysis is required. In particular, we must assume a smallness condition on the coupling constant $\mu$ in relation to $\alpha$ and the Fourier transform of the interaction kernel $k$.

The paper is structured as follows: In Section \ref{sec:prelim} we fix the mathematical notation and state the necessary preliminaries from harmonic analysis. The stabilizability of the linearized system is proven in Section \ref{sec:linear}. The application to the nonlinear system is given in Section \ref{sec:nonlinear}. In section \ref{sec:numerics}, we illustrate the theoretical results with a simulation study. In the final section \ref{sec:discussion}, we situate the obtained results in a wider context and discuss directions for future research.
\section{Preliminaries}\label{sec:prelim}

\subsection{Review of mathematical notation}
As is customary, we denote by $\RR$ and $\CC$ the sets of real and respectively complex numbers. The set of positive integers is denoted by $\NN$. The $d$-dimensional volume (Lebesgue measure) of a set $E \subseteq \RR^d$ is denoted by $|E|$.\footnote{Throughout this work, we do not explicitly state the necessary assumptions of (Lebesgue-) measurability.} By $\1_E$ we denote a function that is $1$ on $E$ and $0$ elsewhere. The support of a function $f:\RR^d \to \CC$ -- the closure of the set of all points on which $f(x) \neq 0$ -- is denoted by $\supp f$. We set
\begin{align*}
	L^p(\RR^d) &= \{f: \|f\|_{L^p(\RR^d)} < \infty \} \\ \|f\|_{L^p(\RR^d)} &= \bigg( \int\limits_{\RR^d} |f|^p\,\drm x \bigg)^{\frac{1}{p}}\,,
\end{align*}
and for $r \in [1,\infty]$ and a (possibly infinite) interval $I \subseteq \RR$ define
\begin{align*}
	L^r(I;L^p(\RR^d)) &=  \{g: \|g\|_{L^r(I;L^p(\RR^d))} < \infty \} \\ \|g\|_{L^r(I;L^p(\RR^d))} &= \bigg( \int\limits_I \bigg( \int\limits_{\RR^d} |g(t;x)|^p\,\drm x \bigg)^{\frac{r}{p}}\,\drm t \bigg)^{\frac{1}{r}}
\end{align*}
with the outer norm replaced by an (essential) supremum if $r = \infty$. For a bounded linear operator $B$ between $L^p$-spaces, we denote the operator norm by $\|B\|$ (or $\|B\|_{L^p \to L^p}$ for emphasis).

By $\cS(\RR^d)$, we denote the Schwartz space of rapidly decreasing, infinitely differentiable functions. The Fourier transform is defined as the operator
\begin{equation*}
	\cS(\RR^d) \to \cS(\RR^d), \quad f \mapsto \hat f(\xi) = \int\limits_{\RR^d} \euler^{-\ii x \cdot \xi} f(x)\,\drm x\,.
\end{equation*}
By extending this operation to the space of tempered distributions $\cS^\prime(\RR^d)$, we may take Fourier transforms of arbitrary $L^p$-functions. We set $D = - \ii \nabla$ and, given $\kappa \in \cS^\prime(\RR^d)$ and $f \in \cS(\RR^d)$, we define $\kappa(D)f$ via $\widehat {\kappa(D)f} = \kappa \hat f$. This notation is motivated by the fact that for polynomial $\kappa$, $\kappa(D)$ acts by differentiation on $f$. Further notation will be introduced as required.

\subsection{Results from harmonic analysis}

Formally, the main geometric assumption on the control set $E$ may be stated as follows:
\begin{definition}\label{def:thick}
	 	Let $L,\rho > 0$. Denote by $\Pi_L$ the cube of sidelength $L$ centered at $0 \in \RR^d$. We say that $E \subseteq \RR^d$ is $(L,\rho)$-thick if
	\[
	\forall x \in \RR^d: \quad |E \cap (\Pi_L + x)| \geq \rho L^d\,.
	\]
	We say that $E$ is \emph{thick} if there exist $L,\rho > 0$ such that $E$ is $(L,\rho)$-thick.
\end{definition}
This notion of thickness is relevant for control theory because of the role it plays in the mathematical (and physical) uncertainty principle. The uncertainty principle (UCP) entails that a function cannot have compact support both in physical and in Fourier space. A sharp quantitative version of the UCP was proven in \cite{Kovrijkine-01}, improving upon the earlier works \cite{Panejah-61, Panejah-62, LogvinenkoS-74}:
\begin{theorem}\label{thm:LS}
	There exists a constant $C_{\mathrm{LS}} \geq 1$ such that for all $p \in [1,\infty]$, all $\lambda > 0$, all $f \in L^p (\RR^d)$ with $\supp \hat f \subseteq \Pi_\lambda$, all $\rho > 0$, all $L > 0$, and all $(\rho , L)$-thick sets $E \subseteq \RR^d$ we have 
	\begin{equation*} 
		\lVert \1_E f \rVert_{L^p (\RR^d)} \geq 
		\left( \frac{\rho}{C_{\mathrm{LS}} } \right)^{C_{\mathrm{LS}} (d + L\cdot \lambda) }  \lVert f \rVert_{L^p (\RR^d)} .
	\end{equation*}
\end{theorem}

It is important to note that this theorem also holds if $\RR^d$ is replaced by a $d$-dimensional torus, as shown in \cite{egidi2018sharp}. This effectively corresponds to the case where $f$ is a periodic, band-limited function.
Finally, we will require a classical multiplier theorem, see for example \cite[Theorem 6.2.7]{grafakos2014classical}. To state it, we need to introduce some further notation: A multi-index $\nu$ is a $d$-dimensional vector with entries in $\NN \cup \{0\}$. Its length $|\nu|$ is the sum of its entries. For $x \in \RR^d$ we write $x^\nu = x_1^{\nu_1} x_2^{\nu_2}\ldots x_d^{\nu_d}$ and $\partial^\nu = \partial_{x_1}^{\nu_1}\ldots \partial_{x_d}^{\nu_d}$\,.
\begin{theorem}\label{thm:HM}
	Let $\kappa:\RR^d \to \RR$ be continuously differentiable up to order $[d/2] + 1$. Assume that there exists a constant $A$, such that for all multi-indices $\nu$ with $|\nu| \leq d/2 + 1$ and all $\xi \in \RR^d$, we have
	\[
		|\xi|^{|\nu|}|\partial^\nu \kappa(\xi)| \leq A\,.
	\]
	For all $p \in (1,\infty)$ there exists a $C_p > 0$ such that
	\[
		\|\kappa(D)\|_{L^p(\RR^d) \to L^p(\RR^d)} \leq C_pA\,.
	\]
\end{theorem}

\section{Stabilizability of the linear equation}\label{sec:linear}
We begin by defining the precise notions of stability that we will consider. Recall that a strongly continuous semigroup $(V_t)_{t \geq 0}$ on  a Banach space $X$ (such as $L^p(\RR^d)$) is a family of bounded linear operators such that $V_0 = \id$, for all $t,s \geq 0$ we have $V_{t+s} = V_t V_s$ and for every $u \in X$, the orbit $t \mapsto V_tu$ is continuous. Each bounded linear operator $B$ generates a strongly continuous semigroup via the exponential map $\euler^{tB}$. In general, strongly continous semigroups may have an unbounded generator.
\begin{definition}
	A strongly continuous semigroup $(V_t)_{t \geq 0}$ is called \emph{exponentially stable} if there exist $M > 0$, $\omega < 0$ such that
	\[
	\forall t > 0: \quad \|V_t\| \leq M\euler^{t\omega}\,.
	\]
\end{definition}
\begin{definition} We say that the system \eqref{eq:amari_lin} is (cost-uniformly) open-loop stabilizable if there exist $m,M \geq 1$ and $\omega < 0$ such that for all $a_0 \in L^p(\RR^d)$, there exists $u \in L^p((0,\infty);L^p(\RR^d))$ with
		\[
		\|u\|_{L^p((0,\infty);L^p(\RR^d))} \leq m\|a_0\|_{L^p(\RR^d)}
		\] such that the solution $a$ of \eqref{eq:amari_lin} satisfies
		\begin{equation*}
			\| a(t) \|_{L^p(\RR^d)} \leq M\euler^{\omega t} \|a_0\|_{L^p(\RR^d)}\,.
		\end{equation*}
		We call $u$ a stabilizing control for $a_0$ and $a$ the associated trajectory. The quantity $\|u\|_{L^p((0,\infty);L^p(\RR^d))}$ is called the control cost.
\end{definition}
\begin{definition}
		We say that the system \eqref{eq:amari_lin} is closed-loop stabilizable if there exists a bounded linear operator $F:L^p(\RR^d) \to L^p(\RR^d)$ -- the feedback -- such that $B + \1_E F$ generates an exponentially stable semigroup.
\end{definition}

The notions of stabilizability for the nonlinear system \eqref{eq:amari} are somewhat more delicate and we postpone a discussion to Section \ref{sec:nonlinear}.

 We begin by analyzing the growth bound of the uncontrolled semigroup $(S_t)_{t \geq 0}$, generated by the operator $B = -\alpha I + \mu K$. Setting $\kappa = \hat k$, we may write $K = \kappa(D)$ and it follows that $S_t = \sigma_t(D)$ where $\sigma_t = \euler^{(-\alpha + \mu\kappa)t}$. In the following, we write
\[
\kappa^* = \sup_{\xi \in \RR^d} \re \kappa(\xi)\,.
\]
The Fourier representation allows us to derive the following growth bound, the proof of which is a standard application of Theorem \ref{thm:HM}.
\begin{lemma} For all $\delta > 0$ there exists $M_\delta \geq 1$ such that
	\[
	\euler^{(\mu\kappa^* - \alpha)t} \leq \|S_t\| \leq M_\delta \euler^{(\mu\kappa^* - \alpha + \delta)t}\,.
	\] 
\end{lemma}
Having established the preliminaries, we are now able to state and prove the main result of this work:
\begin{theorem}\label{thm:lin}
	Let $E$ be a thick set and suppose that $2\alpha - \mu \kappa^* > 0$. Then the system \eqref{eq:amari_lin} is open-loop stabilizable.
\end{theorem}
\begin{proof}
	If $\mu \kappa^* - \alpha < 0$, then the system is trivially stabilizable by the preceding lemma. Therefore, we assume $\mu\kappa^* - \alpha \geq 0$.
	
	Corollary 3.4 in \cite{egidi2024sufficient} states that if there exists a $P \in \cL(L^p(\RR^d))$ such that
	\begin{equation}\label{eq:range_condition}
		\ran(P) \subseteq \ran(P\1_E)\,
	\end{equation}
	and there exist $M \geq 1$ and $\omega_P > \mu \kappa^* - \alpha$ such that
	\begin{equation}\label{eq:dissipation}
		\forall t > 0: \| S_t(\id -P)\| \leq M\euler^{-\omega_P t}\,,
	\end{equation}
	then the system \eqref{eq:amari_lin} is open-loop stabilizable.
	Thus the proof will conclude by choosing a $P$ verifying the range condition \eqref{eq:range_condition} and the dissipation estimate \eqref{eq:dissipation}. We will start with the dissipation estimate.
	
	\paragraph{Dissipation estimate} Denote by $B_r$ the ball of radius $r > 0$. We choose a radially symmetric $\phi \in \cS(\RR^d)$ such that $0 \leq \phi \leq 1$, $\phi = 1$ on the unit ball $B_1$ in $\RR^d$ and $\phi = 0$ on $\RR^d \setminus B_2$. Let $\lambda \geq 1$. We set $\phi_\lambda(x) = \phi(\lambda^{-1} x)$ and $P = \phi_\lambda(D)$. The precise value of $\lambda$ will be chosen later.
	
	The operator $M_t = S_t(\id - P)$ is a Fourier multiplier with symbol \[ m_t(\xi) = \euler^{(-\alpha + \mu \kappa(\xi))t}(1 - \phi(\lambda^{-1} \xi))\,. \] It suffices to show that there exist $M \geq 1, \omega_P > \mu \kappa^* - \alpha$ such that
	\begin{equation}\label{eq:multiplier}
		\| m_t(D) \| \leq M\euler^{-\omega_P t}\,.
	\end{equation}
	For this we'll combine the derivative bounds from \cite{bombach2023observability} with Theorem \ref{thm:HM}: Let $\beta,\gamma \in \NN_0^d$ be multi-indices. Setting
	\[
	C_{\beta} = \max_{\gamma \leq \beta} \sup_{\xi \in \RR^d}|\partial^\gamma \kappa(\xi)|(1 + |\xi|)^{|\beta|}
	\] it follows from the chain rule and the product rule that there exists a constant $K_\beta$ such that
	\begin{align*}
	&(1 + |\xi|)^{|\beta|}|\partial^\beta m_t(\xi)| \\ &\leq K_\beta C_\beta(1 + t|\mu|)^{|\beta|} \euler^{(-\alpha + \mu \re \kappa(\xi))t}\1_{\RR^d \setminus B_\lambda(0)(\xi)}\,.
	\end{align*}
	Since $\kappa(\xi) \to 0$ as $|\xi| \to \infty$, for every $\varepsilon > 0$ we may choose $\lambda$ such that
	\[
	(1 + |\xi|)^{|\beta|}|\partial^\beta m_t(\xi)| \leq K_\beta C_\beta(1 + t|\mu|)^{|\beta|}\euler^{-(\alpha - 2\varepsilon) t}\,.
	\]
	Setting
	\[
	A_{\beta} = K_\beta C_\beta \sup\limits_{t > 0}\, (1 + t|\mu|)^{|\beta|}\euler^{-\varepsilon t}\,,
	\]
	it follows that
	\[
	(1 + |\xi|)^{|\beta|}|\partial^\beta m_t(\xi)| \leq A_\beta\euler^{-(\alpha - \varepsilon)t}\,.
	\]
	By the multiplier theorem \ref{thm:HM}, there exists $M_\varepsilon \geq 1$ such that
	\[
	\|m_t(D)\| \leq M_\varepsilon\euler^{-(\alpha - \varepsilon)t} 
	\]
	By assumption, we may choose $\varepsilon \in (0,2\alpha - \mu\kappa^*)$. Thus
	\[
	(\alpha -\varepsilon) > \alpha - 2\alpha + \mu\kappa^* = \mu\kappa^* - \alpha 
	\]
	which means that \eqref{eq:multiplier} holds with $\omega_P = \alpha -\varepsilon, M = M_\varepsilon$.
	
	\paragraph{Range condition} Choose $q$ such that $\frac{1}{p} + \frac{1}{q} = 1$. Define $Q:L^q(\RR^d) \to L^q(\RR^d)$ via $Q = \phi_\lambda(D)$ with the same $\lambda \geq 1$ as before. Since $E$ is thick, it follows from the Logvinenko-Sereda theorem that there exists a constant $C > 0$ such that for all $q \in (1,\infty)$ we have
	\[
	\forall u \in L^q(\RR^d): \|Qu\|_{L^q(\RR^d)} \leq C\euler^{C\lambda} \|\1_EQu\|_{L^q(\RR^d)}
	\]
	Since $\phi_\lambda$ is radially symmetric, it follows that $Q^* = P$. The range condition is now an easy consequence of a standard duality argument, see \cite[Section 4]{egidi2024sufficient}.
\end{proof}
In order to show closed-loop stabilizability, we need to specialize to the Hilbert space setting $p=2$. In this case, it is an easy consequence of the Riccati theory that the open-loop stabilizability of \eqref{eq:amari_lin} implies the closed-loop stabilizability of \eqref{eq:amari_lin}. To this end, let $a_0 \in L^2(\RR^d)$, and $u$ a stabilizing control of \eqref{eq:amari_lin} with associated trajectory $a$. It follows that the quadratic functional
\[
J(v,b) = \|v\|^2_{L^2((0,\infty);L^2(\RR^d))} + \|b\|^2_{L^2((0,\infty);L^2(\RR^d))}
\]
is finite when evaluated at $v = u, b = a$. This implies the existence of the desired feedback operator $F$ by \cite[Theorem 17.3]{zabczyk2020mathematical}. Thus, we obtain:
\begin{proposition}\label{prop:closed_loop_linear}
	Let $E$ be thick, $p = 2$ and $2\alpha - \mu \kappa^* > 0$. Then the system $\eqref{eq:amari_lin}$ is closed-loop stabilizable.
\end{proposition}

\section{Stabilizability of the nonlinear equation}\label{sec:nonlinear}
In this section, we investigate whether the linear feedback operator $F$ that we derived in the previous section can be used to stabilize the nonlinear system \eqref{eq:amari} to an equilibrium state. We focus here on the situation where the equilibrium state is $0$, but this is mainly for notational convenience. Throughout this section, we assume that the assumptions of Proposition \ref{prop:closed_loop_linear} are fulfilled: $E$ is thick, $2\alpha - \mu \kappa^* > 0$ and $p = 2$. Moreover, we assume the nonlinearity $f$ is globally Lipschitz continuous and that $f(0) = 0, f'(0) = 1$. It is convenient to rewrite the system \eqref{eq:amari} as follows: By the assumptions on $f$, we may write $f(\theta) = \theta + g(\theta)$ with $g(0) = g'(0) = 0$. It is clear that $g$ is Lipschitz continuous as a map $\RR \to \RR$. We define the nonlinear operator
\[
\cN:L^2(\RR^d) \to L^2(\RR^d), \quad \phi \mapsto \mu Kg(\phi)\,.
\]
It is not hard to verify that $\cN$ is globally Lipschitz continuous as a map $L^2(\RR^d) \to L^2(\RR^d)$: Indeed, denoting by $\Lambda$ the (pointwise) Lipschitz constant of $g$, it follows from Young's inequality and the pointwise bound
\[
|g(\theta) -g(\Theta)| \leq \Lambda|\theta - \Theta|, \quad \theta, \Theta \in \RR
\]
that for all $\phi,\psi \in L^2(\RR^d)$, we have
\begin{align}
	\|\cN(\phi) - \cN(\psi)\|_{L^2(\RR^d)} &\leq \mu\|k\|_{L^1(\RR^d)}\|g(\phi) - g(\psi)\|_{L^2(\RR^d)} \notag \\
	&\leq \mu \|k\|_{L^1(\RR^d)}L\|\phi - \psi\|_{L^2(\RR^d)}\,. \label{eq:lipschitz_estimate}
\end{align}
With this notation, the system \eqref{eq:amari} may be rewritten as
\begin{equation*}
	\partial_t a - Ba = \1_Eu + \cN(a), \quad a(0) = a_0 \in L^2(\RR^d)\,.
\end{equation*}
By Proposition \ref{prop:closed_loop_linear}, there exists a feedback operator $F$ such that the semigroup $(V_t)_{t \geq 0}$ generated by $B + \1_EF$ is exponentially stable. Setting $u = Fa_0$, we obtain
\begin{equation}\label{eq:amari_feedback}
	\partial_t a - (B + \1_EF)a = \cN(a), \quad a(0) = a_0 \in L^2(\RR^d)\,. 
\end{equation}
It is a standard result (see \cite[Chapter 6.1]{pazy2012semigroups}), that for every $a_0 \in L^2(\RR^d)$, the system \eqref{eq:amari_feedback} has a unique global mild solution $a \in L^\infty((0,\infty),L^2(\RR^d))$ satisfying
\begin{equation}\label{eq:amari_feedback_mild}
	\forall t \in (0,\infty):\, a(t) = V_ta_0 + \int\limits_0^t V_{t-s}\cN(a(s))\,\drm s\,.
\end{equation}
It can be shown that for a sufficiently weak nonlinearity, the trajectories of the system \eqref{eq:amari_feedback}, obtained from inserting the feedback of the linearized system \eqref{eq:amari_lin} into the nonlinear system \eqref{eq:amari}, are stable. 
\begin{corollary}\label{corr:nonlin}
	Suppose that $\gamma = M\mu \|k\|_{L^1(\RR^d)}\Lambda + \omega < 0$. Then for all $a_0 \in L^2(\RR^d)$ the mild solution of \eqref{eq:amari_feedback} satisfies
	\[
	\forall t \in (0,\infty): \quad \|a(t)\|_{L^2(\RR^d)} \leq M\euler^{\gamma t}\|a_0\|\,.
	\]
\end{corollary}
The proof is a direct consequence of a general result in the stability theory of semilinear systems \cite[Theorem 10.2.2]{cazenave1998introduction} and therefore omitted.

\section{Numerical simulation}\label{sec:numerics}

\begin{figure}[!ht]
	\centering
	\includegraphics[width=0.99\linewidth]{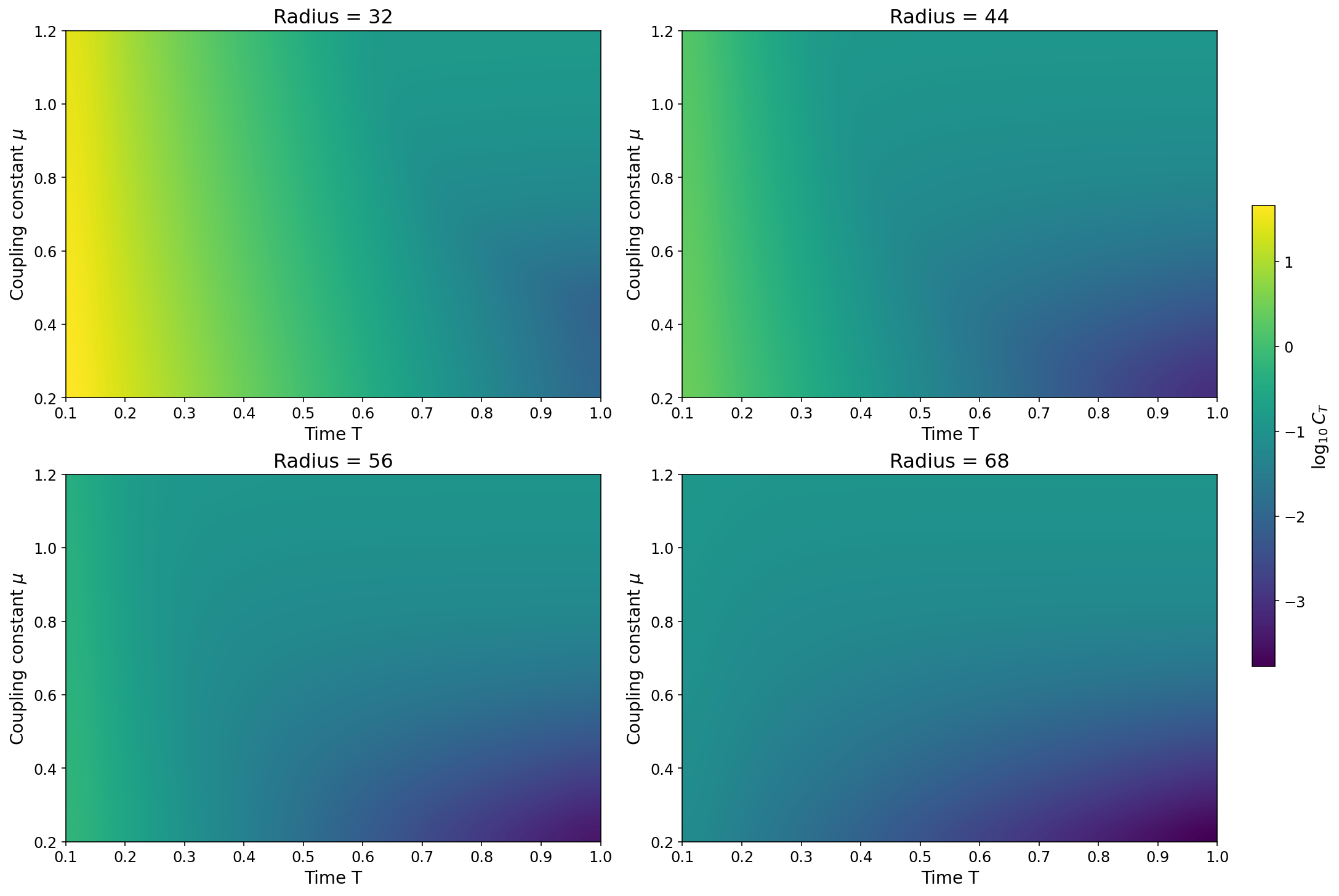}
	\caption{Heatmap showing the results of the parameter sweep. The x-axis represents the final time $T$, and the y-axis the coupling constant $\mu$. As the size of the control set increases and the coupling constant decreases, the observability constant becomes smaller, indicating lower control cost.}
	\label{fig:heatmapobservabilityconstant}
\end{figure}

In this section, we perform a small numerical simulation study to illustrate the results in the linearized setting.  To get an empirical upper bound on the control cost of \eqref{eq:amari_lin}, we calculate the \emph{observability constant} $C_\obs$:
By the duality between controllability (or stabilizability) on the one side and observability on the other, the open-loop stabilizability of \eqref{eq:amari_lin} is equivalent to an observability estimate of the following form: There exist $c \in [0,1), T > 0, C_\obs > 0$ such that for all $a \in L^2(\RR^d)$, we have
\begin{equation}\label{eq:obs}
	\|S_Ta\|^2_{L^2(\RR^d)} \leq  C_\obs\|\1_ES_{(\cdot)}a\|^2_{L^2((0,T);L^2(\RR^d))} + c\|a\|^2_{L^2(\RR^d)}\,.
\end{equation}
If we can choose $c = 0$ in the above inequality, then the system \eqref{eq:amari_lin} is null-controllable in time $T$.

To estimate $C_\obs$ it is useful to introduce the \emph{observability Gramian}, defined as
\[
	G_T = \int\limits_0^T S_t^*\1_E S_t\,\drm t\,.
\]
Here, $S_t^*$ is the adjoint semigroup to $S_t$. Moreover, we set $V_T = S_T^*S_T$. It follows that for $c = 0$, the observability inequality is equivalent to an inequality of the form $V_T \leq C_{\obs} G_T$ in the sense of quadratic forms.

To get a numerical estimate for $C_\obs$, we will work in dimension $d = 1$ on a large interval $[0,N]$ (where $N = 1024$) with periodic boundary conditions. This allows us to represent $V_T,G_T$ as (infinite) matrices in the Fourier basis $\{\euler^{2\pi \ii k x/N}\}_{k = -\infty}^\infty$. We work with finite truncations $X_T = P_M V_T P_M, Y_T = P_M G_T P_M$ where $P_M$ is the projection onto the finite dimensional subspace generated by $\{\euler^{2\pi \ii k x/N}\}_{k = -M}^{M}$. In our experiments, we choose $M = 64$. In this setting, the calculation of $C_\obs$ reduces to finding the largest $\lambda$ solving the generalized eigenvalue problem $X_Tv = \lambda Y_Tv$ with an eigenvector $v \in \RR^{2M+1} \setminus \{0\}$. One then has $C_\obs \approx \sqrt{\lambda}$. For the numerical implementation, we employ the SciPy eigensolver \cite{2020SciPy-NMeth}.

The control set $E$ is given as the union of $n=10$ randomly chosen sub-intervals of varying length $2r$. We fix a constant decay rate $\alpha = 9$ and as the kernel function $k$, we pick a Mexican hat kernel $k(x) = (1 - (x/64)^2)\euler^{-(x/64)^2/2}$. The Fourier transform is $\kappa(\xi) = C\xi^2\euler^{-(64\xi)^2/2}$ where $C$ is some normalization constant. The coupling parameter $\mu$ is left to vary together with the final time $T$ and the interval radius $r$. Thereby, we aim to estimate the effect of the interaction term and the size of the control set on the final-state observability of the system \eqref{eq:amari_lin}.

The results of the parameter sweep are depicted in Fig. \ref{fig:heatmapobservabilityconstant}. Since the observability constant depends on the constant from Theorem \ref{thm:LS} and the growth bound of the semigroup, the theory predicts that the observability constant should decrease as the size of the control set increases and the coupling constant $\mu$ decreases. Thus, the figure shows that the results from the numerical experiment are in line with the predictions and that the observability, and therefore the controllability and stabilizability of the system are highly sensitive to the geometry of the control set.

\section{Discussion}\label{sec:discussion}
We have shown that neural fields are stabilizable from thick subsets of $\RR^d$. More precisely, we have shown that there are three cases that have to be distinguished, depending on the relation between the activity decay $\alpha$ and the terms $\mu, \kappa^*$, coming from the neural interaction. If $ \mu\kappa^* - \alpha < 0$, then the linear system is already stable and there is no need for stabilization. If $ \mu\kappa^* - \alpha > 0$ but $\mu\kappa^* - 2\alpha < 0$, then stabilization is possible by an open-loop control (Theorem \ref{thm:lin}). If $\mu\kappa^* - 2\alpha > 0$, it is not known whether the system is stabilizable to an equilibrium state. This result stands in stark contrast to the one of \cite{tamekue2025control}. There, the authors prove approximate controllability to arbitrary target states with an error of size $O(T^2)$ as $T \to 0$, under a mild spectral condition on the Fr\'echet derivative of the nonlinearity, but with the assumption that $E = \RR^d$. Under restrictive smallness conditions on $\mu$ and the final time $T$, they even obtain exact controllability. This suggests that restricting the control set either puts a strong theoretical limit on controllability of neural fields, or requires more advanced methods to show even approximate controllability.

Let us note that we have made various simplifying assumptions throughout this work: For instance, we have assumed that there is only one population of neurons. More biologically plausible neural field models often involve multiple populations of excitatory or inhibitory neurons. Furthermore, we have assumed that the activation occurs without delay. Again, it would increase the biological realism to incorporate a delay, as done for instance in \cite{brivadis2024adaptive}.

Regarding possible ways to handle multiple populations, we note that a Lebeau-Robbiano strategy for coupled heat-like systems has already been implemented in \cite{bombach2025logvinenko}. It seems plausible that such an approach would extend to linearized Amari systems involving multiple populations by the method outlined here. Incorporating delay would be more challenging within our framework and require an understanding of how to implement the Lebeau-Robbiano strategy in this setting or the pursuit of an alternate approach, for instance via Carleman estimates.

On the practical side, our numerical experiment indicates that the size of the control set plays a substantial role in the determination of the control cost. On the theoretical side, our work demonstrates the utility of the Lebeau-Robbiano strategy even in situations where the underlying equation is quite different from the heat equation for which this strategy was originally intended. Indeed, the heat equation is local, linear and smoothing while we are here dealing with a nonlocal, nonlinear and nonsmoothing equation. It is known from the work \cite{miller2006controllability}, that for fractional heat equations of the form $\partial_t a - (-\Delta)^s a = \1_E u$, controllability is only guaranteed if $s > 1/2$ and fails if $s \leq 1/2$ . However, stabilizability continues to hold as long as $s > 0$ (see \cite{egidi2024sufficient}). Morally speaking, the Amari-type field equation roughly corresponds to the regime where $(-\Delta)^s$ is replaced by $\phi(-\Delta)$ with $\phi$ bounded. Thus, it appears that we are at the limit of what is possible with a straightforward application of the Lebeau-Robbiano strategy. 

We conclude that it is a challenging open problem to determine the sharp geometric condition on which approximate controllability of neural fields holds, in particular in the strong-coupling regime $\mu\kappa^* - 2\alpha > 0$. As an intermediate step, it might be of interest to consider sets that satisfy the stronger Geometric Control Condition (GCC), introduced in \cite{bardos1987controle} to study the null-controllability of the wave equation. This choice is natural since the wave equation, like the Amari equation, preserves smoothness but is not (globally) smoothing.

\bibliographystyle{IEEEtran}
\bibliography{lit}

@ARTICLE{2020SciPy-NMeth,
	author  = {Virtanen, Pauli and Gommers, Ralf and Oliphant, Travis E. and
	Haberland, Matt and Reddy, Tyler and Cournapeau, David and
	Burovski, Evgeni and Peterson, Pearu and Weckesser, Warren and
	Bright, Jonathan and {van der Walt}, St{\'e}fan J. and
	Brett, Matthew and Wilson, Joshua and Millman, K. Jarrod and
	Mayorov, Nikolay and Nelson, Andrew R. J. and Jones, Eric and
	Kern, Robert and Larson, Eric and Carey, C J and
	Polat, {\.I}lhan and Feng, Yu and Moore, Eric W. and
	{VanderPlas}, Jake and Laxalde, Denis and Perktold, Josef and
	Cimrman, Robert and Henriksen, Ian and Quintero, E. A. and
	Harris, Charles R. and Archibald, Anne M. and
	Ribeiro, Ant{\^o}nio H. and Pedregosa, Fabian and
	{van Mulbregt}, Paul and {SciPy 1.0 Contributors}},
	title   = {{{SciPy} 1.0: Fundamental Algorithms for Scientific
	Computing in Python}},
	journal = {Nature Methods},
	year    = {2020},
	volume  = {17},
	pages   = {261--272},
	adsurl  = {https://rdcu.be/b08Wh},
	doi     = {10.1038/s41592-019-0686-2},
}

@article{amari1977dynamics,
	title={Dynamics of pattern formation in lateral-inhibition type neural fields},
	author={Amari, Shun'ichi},
	journal={Biological cybernetics},
	volume={27},
	number={2},
	pages={77--87},
	year={1977},
	publisher={Springer}
}

@article{bardos1987controle,
	title={Contr{\^o}le et stabilisation pour l'{\'e}quation des ondes},
	author={Bardos, Claude and Lebeau, Gilles and Rauch, Jeff},
	journal={Journ{\'e}es {\'e}quations aux d{\'e}riv{\'e}es partielles},
	pages={1--15},
	year={1987}
}

@article{bombach2023observability,
	title={Observability and null-controllability for parabolic equations in {$L_p$}-spaces},
	author={Bombach, Clemens and Gallaun, Dennis and Seifert, Christian and Tautenhahn, Martin},
	journal={Mathematical Control and Related Fields},
	volume={13},
	number={4},
	pages={1484--1499},
	year={2023},
	publisher={Mathematical Control and Related Fields}
}

@article{bombach2025logvinenko,
	author = {Bombach, Clemens and Tautenhahn, Martin},
	title = {A {Logvinenko--Sereda} theorem for vector-valued functions and application to control theory},
	journal = {Z. Anal. Anwend.},
	volume = {44},
	number = {3/4},
	pages = {323--354},
	year = {2025}
}

@article{bolelli2025neural,
	title={Neural field equations with time-periodic external inputs and some applications to visual processing},
	author={Bolelli, M. Virginia and Prandi, Dario},
	journal={Journal of Mathematical Imaging and Vision},
	volume={67},
	number={4},
	pages={47},
	year={2025},
	publisher={Springer}
}

@article{brivadis2024adaptive,
	title={Adaptive observer and control of spatiotemporal delayed neural fields},
	author={Brivadis, Lucas and Chaillet, Antoine and Auriol, Jean},
	journal={Systems \& Control Letters},
	volume={186},
	pages={105777},
	year={2024},
	publisher={Elsevier}
}

@book{cazenave1998introduction,
	title={An introduction to semilinear evolution equations},
	author={Cazenave, Thierry and Haraux, Alain},
	volume={13},
	year={1998},
	publisher={Oxford University Press}
}

@article{chaillet2017robust,
	title={Robust stabilization of delayed neural fields with partial measurement and actuation},
	author={Chaillet, Antoine and Detorakis, Georgios Is and Palfi, St{\'e}phane and Senova, Suhan},
	journal={Automatica},
	volume={83},
	pages={262--274},
	year={2017},
	publisher={Elsevier}
}

@book{coombes2014neural,
	title={Neural fields: theory and applications},
	author={Coombes, Stephen and beim Graben, Peter and Potthast, Roland and Wright, James},
	year={2014},
	publisher={Springer},
	doi = {10.1007/978-3-642-54593-1}
}

@article{cook2022neural,
	title={Neural Field Models: A mathematical overview and unifying framework},
	author={Cook, Blake J and Peterson, Andre DH and Woldman, Wessel and Terry, John R},
	journal={Mathematical Neuroscience and Applications},
	volume={2},
	year={2022},
	publisher={Episciences. org},
	doi={mna.episciences.org/9228}
}

@article{egidi2018sharp,
	title={Sharp geometric condition for null-controllability of the heat equation on R d and consistent estimates on the control cost},
	author={Egidi, Michela and Veseli{\'c}, Ivan},
	journal={Archiv der Mathematik},
	volume={111},
	number={1},
	pages={85--99},
	year={2018},
	publisher={Springer}
}

@article{egidi2024sufficient,
	title={Sufficient criteria for stabilization properties in Banach spaces},
	author={Egidi, Michela and Gallaun, Dennis and Seifert, Christian and Tautenhahn, Martin},
	journal={Integral Equations and Operator Theory},
	volume={96},
	number={2},
	pages={13},
	year={2024},
	publisher={Springer}
}

@article{gallaun2020sufficient,
	title={Sufficient criteria and sharp geometric conditions for observability in Banach spaces},
	author={Gallaun, Dennis and Seifert, Christian and Tautenhahn, Martin},
	journal={SIAM journal on control and optimization},
	volume={58},
	number={4},
	pages={2639--2657},
	year={2020},
	publisher={SIAM}
}

@book{grafakos2014classical,
	title={Classical {F}ourier analysis},
	author={Grafakos, Loukas},
	volume={3},
	year={2014},
	publisher={Springer}
}

@article{gu2015controllability,
	title={Controllability of structural brain networks},
	author={Gu, Shi and Pasqualetti, Fabio and Cieslak, Matthew and Telesford, Qawi K and Yu, Alfred B and Kahn, Ari E and Medaglia, John D and Vettel, Jean M and Miller, Michael B and Grafton, Scott T and others},
	journal={Nature communications},
	volume={6},
	number={1},
	pages={8414},
	year={2015},
	publisher={Nature Publishing Group UK London}
}

@article{karrer2020practical,
	title={A practical guide to methodological considerations in the controllability of structural brain networks},
	author={Karrer, Teresa M and Kim, Jason Z and Stiso, Jennifer and Kahn, Ari E and Pasqualetti, Fabio and Habel, Ute and Bassett, Danielle S},
	journal={Journal of neural engineering},
	volume={17},
	number={2},
	pages={026031},
	year={2020},
	publisher={IOP Publishing}
}

@Article{Kovrijkine-01,
	author  = {O. Kovrijkine},
	title   = {Some results related to the {L}ogvinenko-{S}ereda theorem},
	number  = {10},
	pages   = {3037--3047},
	volume  = {129},
	journal = {Proc. Amer. Math. Soc.},
	year    = {2001},
}

@Article{LogvinenkoS-74,
	author  = {{V.~N.} Logvinenko and {Ju.~F.} Sereda},
	title   = {Equivalent norms in spaces of entire functions of exponential type},
	pages   = {102--111},
	volume  = {20},
	journal = {Teor. Funkts., Funkts. Anal. Prilozh.},
	year    = {1974},
}

@article{lebeau1995controle,
	title={Contr{\^o}le exact de l'{\'e}quation de la chaleur},
	author={Lebeau, Gilles and Robbiano, Luc},
	journal={Communications in Partial Differential Equations},
	volume={20},
	number={1-2},
	pages={335--356},
	year={1995},
	publisher={Taylor \& Francis}
}

@article{miller2006controllability,
	title={On the controllability of anomalous diffusions generated by the fractional Laplacian},
	author={Miller, Luc},
	journal={Mathematics of Control, Signals and Systems},
	volume={18},
	number={3},
	pages={260--271},
	year={2006},
	publisher={Springer}
}

@article{miller2010direct,
	title={A direct {L}ebeau-{R}obbiano strategy for the observability of heat-like semigroups},
	author={Miller, Luc},
	journal={Discrete and Continuous Dynamical Systems-Series B},
	volume={14},
	number={4},
	pages={1465--1485},
	year={2010}
}

@article{Panejah-61,
	title={Some theorems of {P}aley-{W}iener type},
	author={{B. P.} Panejah},
	journal={Soviet Math. Dokl.},
	volume={2},
	number={},
	pages={533--536},
	year={1961},
}

@article{Panejah-62,
	title={On some problems in harmonic analysis},
	author={{B. P.} Panejah},
	journal={Dokl. Akad. Nauk SSSR},
	volume={142},
	pages={1026--1029},
	year={1962},
}

@book{pazy2012semigroups,
	title={Semigroups of linear operators and applications to partial differential equations},
	author={Pazy, Amnon},
	year={1983},
	publisher={Springer Science \& Business Media}
}

@article{tamekue2025control,
	title={Control of neural field equations with step-function inputs},
	author={Tamekue, Cyprien and Ching, ShiNung},
	journal={arXiv preprint arXiv:2510.22022},
	year={2025}
}

@article{tamekue2024mathematical,
	title={A mathematical model of the visual {M}ac{K}ay effect},
	author={Tamekue, Cyprien and Prandi, Dario and Chitour, Yacine},
	journal={SIAM Journal on Applied Dynamical Systems},
	volume={23},
	number={3},
	pages={2138--2178},
	year={2024},
	publisher={SIAM},
	doi={10.1137/23M1616686}
}

@article{taylor2014computational,
	title={A computational study of stimulus driven epileptic seizure abatement},
	author={Taylor, Peter Neal and Wang, Yujiang and Goodfellow, Marc and Dauwels, Justin and Moeller, Friederike and Stephani, Ulrich and Baier, Gerold},
	journal={PLOS one},
	volume={9},
	number={12},
	pages={e114316},
	year={2014},
	publisher={Public Library of Science San Francisco, USA}
}

@book{zabczyk2020mathematical,
	title={Mathematical control theory},
	author={Zabczyk, Jerzy},
	year={2020},
	publisher={Springer}
}

\end{document}